\documentclass[preprint,11pt]{elsarticle}
\makeatletter
\def\ps@pprintTitle{%
 \let\@oddhead\@empty
 \let\@evenhead\@empty
 \def\@oddfoot{\centerline{\thepage}}%
 \let\@evenfoot\@oddfoot}
\makeatother
\usepackage{amssymb}
\usepackage{amsmath}
\usepackage{amsthm}
\usepackage{epsfig}
\usepackage{color}
\usepackage{makeidx}
\usepackage{bbm}
\usepackage[top=1in, bottom=1in, inner=1in, outer=1in]{geometry}
\usepackage{setspace}

\newtheorem{theorem}{Theorem}
\newtheorem{lemma}{Lemma}
\newtheorem{proposition}{Proposition}
\newtheorem*{remark}{Remark}

\newcommand*\xbar[1]{%
  \hbox{%
    \vbox{%
      \hrule height 0.5pt % The actual bar
      \kern0.4ex%         % Distance between bar and symbol
      \hbox{%
        \kern-0.15em%      % Shortening on the left side
        \ensuremath{#1}%
        \kern-0.15em%      % Shortening on the right side
      }%
    }%
  }%
}

\begin{document}

\begin{frontmatter}

\title{Conditional Speed of Branching Brownian Motion, Skeleton Decomposition and
Application to Random Obstacles}

\author[koc,ozyegin]{Mehmet \"{O}z}
\ead{mehmet.oz@ozyegin.edu.tr}

\author[koc]{Mine \c{C}a\u{g}lar}
\ead{mcaglar@ku.edu.tr}

\author[colorado]{J\'{a}nos Engl\"{a}nder}
\ead{janos.englander@colorado.edu}

\address[koc]{Department of Mathematics, Ko\c{c} University, Istanbul, Turkey}
\address[colorado]{Department of Mathematics, University of Colorado at
Boulder, Boulder, CO-80309, USA}
\address[ozyegin]{Department of Natural and Mathematical Sciences, Faculty of Engineering, \"{O}zye\u{g}in University, Istanbul, Turkey}

\begin{abstract}
We study a branching Brownian motion $Z$ in $\mathbb{R}^d$, among obstacles scattered
according to a Poisson random measure with a radially decaying intensity. Obstacles
are balls with constant radius and each one works as a trap for the whole motion when
hit by a particle. Considering a general offspring distribution, we derive the decay
rate of the annealed probability that none of the particles of $Z$ hits a trap,
asymptotically in time $t$. This proves to be a rich problem motivating the proof of
a more general result about the speed of branching Brownian motion conditioned on
non-extinction. We provide an appropriate ``skeleton" decomposition for the underlying
Galton-Watson process when supercritical and show that the ``doomed" particles do not contribute to the asymptotic decay rate.
\end{abstract}

\vspace{3mm}

\begin{keyword}
Branching Brownian motion \sep Poissonian traps \sep Random environment  \sep Hard
obstacles \sep Rightmost particle
\vspace{3mm}
\MSC[2010] 60J80 \sep 60K37 \sep 60F10
\end{keyword}

\end{frontmatter}

\pagestyle{myheadings}
\markright{CONDITIONAL SPEED OF BBM AND RANDOM OBSTACLES\hfill}

\section{Introduction}

Branching Brownian motion (BBM) among random obstacles functioning as traps has been studied recently in \cite{E2000,E2003,E2008,LV2009,OC2013}. We study a BBM that evolves in $\mathbb{R}^d$, where a radially decaying field of Poissonian traps is present. We are mainly interested in the trap-avoiding probability of the system up to time $t$, where the underlying Galton-Watson process (GWP) has a general offspring distribution. Investigation of this problem leads us to finding the speed of BBM conditioned on non-extinction, which is of independent interest.

Let $Z=(Z(t))_{t\geq 0}$ be a $d$-dimensional BBM with branching rate $\beta>0$ and offspring
distribution $\lambda$. The process starts with a single particle at the origin,
which performs a Brownian motion in $\mathbb{R}^d$ for a random time which is
distributed exponentially with constant parameter $\beta$. Then, the particle dies
and simultaneously gives birth to a random number of particles distributed according
to the offspring distribution $\lambda$, which is a probability measure on
$\mathbb{N}=\left\{0,1,2,\ldots\right\}$. Similarly, each offspring particle repeats
the same procedure independently of all others, starting from the position of her
parent. In this way, one obtains a measure-valued Markov process $Z=(Z(t))_{t\geq0}$,
where $Z(t)$ can be identified as a particle configuration for given $t\geq 0$. By assumption, $Z(0)=\delta_0$.
The total mass process $|Z|=(|Z(t)|)_{t\geq 0}$ is a continuous time branching process
with rate $\beta$, and the number of particles in generation $n$ of $|Z|$ is a GWP with offspring distribution $\lambda$. The initial particle present at $t=0$ constitutes the $0$th generation, the offspring of the initial particle constitute the $1$st generation, and so forth. We denote the extinction time of the process $|Z|$ by $\tau$, which is formally defined as
$\tau=\inf\left\{t\geq 0:|Z(t)|=0\right\}$, where we use the convention that $\inf\emptyset=\infty$. We then denote the event of extinction of the process $|Z|$ by $\mathcal{E}$, and formally write $\mathcal{E}=\left\{\tau<\infty \right\}$. We use the term \textit{non-extinction} for the event $\mathcal{E}^c$. Let $P$ be the probability corresponding to the process $Z$, and $E$ the corresponding expectation.

The branching Brownian motion is assumed to live in a random environment consisting
of Poissonian traps. Let $\mathbb{B}(\mathbb{R}^d)$ be the $d$-dimensional Borel
sets, and let $\Pi$ denote the Poisson random measure on $\mathbb{B}(\mathbb{R}^d)$,
with a spatially dependent locally finite mean measure $\nu$ such that
$\text{d}\nu/\text{d}x$ exists and is continuous on $\mathbb{R}^d$ and
\begin{equation} \frac{\text{d}\nu}{\text{d}x} \sim \frac{l}{|x|^{d-1}}, \quad
|x|\rightarrow\infty, \quad l>0, \label{eq0}
\end{equation}
where $\text{d}x$ is for the Lebesgue measure. A random trap configuration $K$ with
radius $a$ on $\mathbb{R}^d$ is defined as
\begin{equation}
K=\underset{x_i \in
\text{supp}(\Pi)}{\bigcup} \bar{B}(x_i,a), \label{eq00}
\end{equation}
where $\bar{B}(x_i,a)$ is the closed ball of radius $a>0$ centered at $x_i$. Let
$\mathbb{P}$ be the probability for the Poisson random measure, and $\mathbb{E}$ the
corresponding expectation.

For $A\in \mathbb{B}(\mathbb{R}^d)$ and $t\geq 0$, let $|\text{supp}(Z(t))\cap A|$ be the number
of particles located in $A$ at time $t$. For $t\geq 0$, let
\begin{equation} R(t)=\underset{s \in\text{[0,t]}}{\bigcup} \text{supp}(Z(s)) \label{range}
\end{equation}
be the range of $Z$ up to time $t$. Now let $T$ be the first time that $Z$ hits a
trap,
\begin{equation} T=\inf \left\{t \geq 0:|\text{supp}(Z(t))\cap K|>0\right\}=\inf \left\{t
\geq 0:R(t)\cap K \neq \emptyset \right\}.  \nonumber
\end{equation}
Then, the event of trap-avoiding up to time $t$ of $Z$ among the Poissonian traps is
given by $\left\{T>t\right\}$.

%%%%%%%%%%%%%%%%%%%%%% Updated by Mine %%%%%%%%%%%%%%%%%%%%%%%%%%%%%%%%
We aim to analyze $(\mathbb{E}\times P)(T>t)$, the annealed (averaged) trap-avoiding probability of the system up to time $t$, for a BBM with a general offspring distribution $\lambda$. In \cite{E2003}, the asymptotic decay of the annealed probability has been found as $t\rightarrow\infty$ in the strictly dyadic branching case, that is, $\lambda(2)=1$. On the way to its generalization, we find out that the problem proves to be rich enough to require several stand-alone results that we also contribute in this paper. In particular, we find the speed of BBM conditioned on non-extinction as given in Theorem~\ref{teo1}. The reader may have the feeling that the result is expected and even `folklore,' however we were unable to locate the result and its proof in the existing literature, and feel that it is useful to state and prove it here in a self-contained way. When the GWP is supercritical, the particles are grouped into those with infinite or finite line of descent, so-called ``skeleton" and ``doomed" particles, respectively. In other words, a skeleton decomposition is performed to analyze the problem. Our analysis indicates that the doomed particles do not contribute to the conditional speed. Here, the `speed' of BBM refers to the growth rate of the radius of the minimal ball containing the range of the BBM. It has first been studied by McKean \cite{MK1975,MK1976}, later by Bramson \cite{B1978,B1983} and Chauvin and Rouault \cite{CR1988}, and is still subject to active research.

In Theorem~\ref{teo2}, we prove the asymptotic decay of the annealed trap-avoiding probability as a large deviation result. It is an important and non-trivial application of Theorem~\ref{teo1}. The problem of trap-avoiding asymptotics for BBM among Poissonian traps has been previously studied by Engl\"{a}nder \cite{E2000} for the case $d\geq 2$ and where the trap intensity was uniform. Then, in search for extending the result to the case $d=1$, Engl\"{a}nder and den Hollander \cite{E2003} considered the more interesting case where the trap intensity was radially decaying as given in (\ref{eq0}). Both \cite{E2000,E2003} assumed  strictly dyadic underlying GWP (precisely two offspring).

The reason the decay rate given in \eqref{eq0} is the `interesting' one is that it is in fact the `borderline' one. This is explained in Theorem 1.3 in \cite{E2003}, which describes the optimal survival strategy, as it depends on the `fine tuning constant' $\ell$ (we use  $l$ instead of $\ell$ in the present paper). Namely, it was shown that 
\begin{itemize}
\item[--]
In the low intensity regime $\ell<\ell_{cr}$, the system clears a ball of
radius $\sqrt{2\beta}\,t$ from traps, and until time $t$ stays inside this ball and
branches at rate $\beta$.

\item[--]
In the high intensity regime $\ell>\ell_{cr}$:
\begin{itemize}
\item[$d=1$:]
The system clears an $o(t)$-ball (i.e., a ball with radius $>a$ but $\ll t$),
and until time $t$ suppresses the branching (i.e., produces a polynomial number
of particles) and stays inside this ball.

\item[$d\geq 2$:]
The system clears a ball of radius $\sqrt{2\beta}\,(1-\eta^*)t$
around a point at distance $c^*t$ from the origin, suppresses the
branching until time $\eta^*t$,  and during the remaining time
$(1-\eta^*)t$  branches at rate $\beta$.

\end{itemize}
\end{itemize}
(See Theorem 1.3 in \cite{E2003} for the precise statements.)

Hence, the decay considered is indeed the `borderline' one, where the behavior of the system, conditioned on survival up to $t$, depends only on the constant $\ell$, and exhibits a change of behavior at the crossover. If one considers a larger (smaller) decay order, the optimal strategy will simply follow the one exhibited when the decay is as in \eqref{eq0} and $\ell>\ell_{cr}$ ($\ell<\ell_{cr}$); although if the decay order is very large, then $\eta^*=1$  (complete suppression of branching) may occur even for $d\ge 2$, while $0<\eta^*<1$ is always the case in the high intensity regime studied in \cite{E2003}.

This `borderline property' of the decay in \eqref{eq0} will remain the case in our paper, the main difference being that we consider a general offspring distribution, and thus we have to take into account that extinction for the GWP may have positive probability.

In \"{O}z and \c{C}a\u{g}lar \cite{OC2013} already a general offspring distribution  was studied, albeit with uniform Poisson intensity. In the present work, we consider a radially decaying trap intensity as in (\ref{eq0}), and a general offspring distribution for the underlying GWP. We show that, conditioned on non-extinction, the doomed particles of the decomposition of the supercritical GWP do not contribute to the asymptotic decay rate of the trap-avoiding probability. In comparison with the strictly dyadic case, the mean and another functional of the offspring distribution appear as parameters in the rate expression (see (\ref{I})). We refer the reader to \cite{E2007} for an interesting survey article on the topic of BBM among Poissonian traps, and to \cite{E2008,LV2009} for various related problems.

Theorem~\ref{teo3} is a technical result identifying the rate function of Theorem~\ref{teo2} through tedious analysis.

The annealed setting is interesting in that it provides a connection between the trapping problem and the `branching Brownian sausage' (see \cite{E2000}).
If $Z_t^a$ denotes the $a$-neighborhood ($a>0$), of the range $R(t)$:
\[
Z_t^a := \underset{x\in R(t)}{\bigcup} \bar{B}(x,a)\, ,
\]
then, using the analogy of the well-known `Wiener-sausage', $Z_t^a$ describes the `sausage' for the branching process up to $t$, called the branching Brownian sausage. The connection with the annealed obstacle problem is that
\[
E\,e^{-\nu(Z_t^a)}=(\mathbb{E}\times P)(T>t)\, ,
\]
or, alternatively, here $P,E$ can be replaced by $P(\cdot | {\cal E}^c)$, $E(\cdot | {\cal E}^c)$, too.

If $d=1$ and $d\nu /dx$ is precisely $l$, then we get
\[
E\, e^{-l|Z_t^a|}=(\mathbb{E}\times P)(T>t)\, ,
\]
where $|Z_t^a|$ is the volume of the sausage (or with $P(\cdot | {\cal E}^c)$). Since $d=1$, $Z_t^a$ is the same as the interval $(m_t -a, M_t+a)$, where $m_t:=\min_{x\in R(t)} \{x\}$ and $M_t:=\max_{x\in R(t)} \{x\}$. So, our asymptotics describes the large deviations for $M_t-m_t$. If $d\ge 2$, then the difference is that instead of the volume $|Z_t^a|$, it yields a result on $\nu(Z_t^a)$ (which is a kind of `Riesz-potential').

We have studied the first time a particle of BBM hits a trap rather than the time $\widetilde{T}$ that all particles are absorbed by the traps. The reason is that there is no tail for it is known that $\underset{t\rightarrow\infty}{\lim}(\mathbb{E}\times P)(\widetilde T>t)>0$  as shown in \cite[Thm.5.4]{E2015}. Heuristically, this follows from the fact that the system may survive by suppressing the traps in the ball $B(0,R)$ with $R>R_0$, where $B(0,R)$ is the open ball of radius $R$ centered at the origin, and $R_0$ is chosen such that even the branching process with killing at $\partial B(0,R)$ may survive forever.

The organization of the paper is as follows. In Section 2, we give the statements of Theorem~\ref{teo2} and Theorem~\ref{teo3} as our main results. Section 3 includes  the essential lemmas together with their proofs. In Section 4, we state and prove Theorem 3. In sections 5 and 6, we prove Theorems 1 and 2, respectively.
%%%%%%%%%%%%%%%%%%%End of update by Mine%%%%%%%%%%%%%%%%%%%%%%%%%%%%%%%%%

\section{Main Results}
To formulate our main results, we introduce further notation. Let $f$ be the probability generating function (p.g.f.) of the offspring distribution, and $\mu$ be the mean number of offspring,
\[f(s)=\sum_{j=0}^{\infty} \lambda(j) s^j,
\]
\[\mu=\sum_{j=0}^{\infty}j \lambda(j),
\]
and define
\[m=\mu-1.
\]
Throughout this work, we assume that $\mu<\infty$ and without loss of generality that $\lambda(1)=0$ (see the proof of Lemma~\ref{decomp}).
Now let $q=P(\mathcal{E})$ be the probability of extinction for the underlying GWP, and let
\[\alpha=1-f'(q)
\]
Note that $\lambda(0)=0$ implies that $q=0$ since two or more offspring is produced each time a particle branches. When $\lambda(0)=0$, it then follows that $\alpha=1$ as $f'(0)=\lambda(1)$, and $\lambda(1)=0$ by assumption. Also, $q=1$ when $\mu \leq 1$, and $q\in(0,1)$ when $\lambda(0)>0$ and $\mu>1$.

For $r,b \geq 0$, define
\begin{equation} g_d(r,b)=\int_{B(0,r)} \frac{\text{d}x}{|x+be|^{d-1}}, \nonumber
\end{equation}
where $e=(1,0,\ldots,0)$ is the unit vector in the direction of the first coordinate. Next, let
\begin{equation} l_{cr}^*=l_{cr}^*(m,\beta,d)=\frac{1}{s_d}\sqrt{\frac{\beta}{2m}}, \label{lcr}
\end{equation}
where $s_d$ is the surface area of the d-dimensional unit ball ($s_1=2$, $s_2=2\pi$, $s_3= 4\pi$, etc.).

We now state our main results. Recall that $l$ is the constant in the definition of $\nu$ in (\ref{eq0}), and $a$ is the constant trap radius given in (\ref{eq00}).
\begin{theorem}[Variational formula] \label{teo2}
Fix $d$,$f$,$\beta$,$a$. Define

\begin{equation} I(l,f,\beta,d)=\underset{\eta \in [0,1],c \in [0,\sqrt{2\beta}]}{\text{min}} \left\{\beta\alpha\eta+\frac{c^2}{2\eta}+l g_d(\sqrt{2\beta m}(1-\eta),c)\right\}. \label{I}
\end{equation}
(For $\eta=0,c=0$, set $c^2/2\eta=0$, and for $\eta=0,c>0$, set $c^2/2\eta=\infty$.)

\begin{enumerate}
\item If $\lambda(0)=0$, then
\begin{equation} \underset{t \rightarrow \infty}{\lim} \frac{1}{t}\log (\mathbb{E}\times P) \left(T>t\right) = -I(l,f,\beta,d).  \nonumber
\end{equation}

\item If $\lambda(0)>0$ and $m>0$, then
\begin{equation} \underset{t \rightarrow \infty}{\lim} \frac{1}{t}\log (\mathbb{E}\times P) \left(T>t|\mathcal{E}^c\right) = -I(l,f,\beta,d). \nonumber
\end{equation}

\item If $\lambda(0)>0$ and $m\leq 0$, then
\begin{equation} \underset{t \rightarrow \infty}{\lim} (\mathbb{E}\times P) \left(T>t\right)=(\mathbb{E}\times P) \left(T>\tau\right)>0,  \nonumber
\end{equation}
where $\tau=\inf\left\{t\geq 0:|Z(t)|=0\right\}$ is the extinction time  and $\mathcal{E}$ is the event of extinction for the BBM.

\end{enumerate}
\end{theorem}

In comparison with the variational result given in \cite[Theorem 1.1]{E2003} for dyadic
branching, Theorem~\ref{teo2} contributes two extra factors, both naturally
depending only on $f$, in the expression for the rate function (\ref{I}).
The extra factor $\alpha$ in the first term appears
when $\lambda(0)>0$ (if $\lambda(0)=0$, then $\alpha=1$). It is a consequence
of the modified p.g.f. for the offspring distribution of the skeleton
particles, which are part of the decomposition of the supercritical GWP
conditioned on non-extinction. The extra factor $m$ inside $g_d$ appears,
because the speed of BBM is now no longer $\sqrt{2\beta}$ but $\sqrt{2\beta
m}$. Since $\alpha=1=m$ for strictly dyadic branching, the result in \cite{E2003} is
recovered by (\ref{I}) as a special case.

The next result is purely analytic, and solves the variational problem for the rate function $I(l,f,\beta,d)$. We generalize the corresponding result in \cite{E2003}.

\begin{theorem}[Crossover] \label{teo3}
Fix $f$,$\beta$,$a$.

\begin{enumerate}

\item For $d\geq 1$ and all $l\neq l_{cr}$, where $l_{cr}$ is defined below, the variational problem has a unique pair of minimizers, denoted by $\eta^*=\eta^*(l,f,\beta,d)$ and $c^*=c^*(l,f,\beta,d)$.

\vspace{5mm}

\item
For $d=1$,
\begin{align}  &l\leq l_{cr}:\quad I(l,f,\beta,d)=\beta \frac{l}{l_{cr}^*}, \nonumber \\ &l>l_{cr}:\quad I(l,f,\beta,d)=\beta \alpha, \label{eq51}
\end{align}
and
\begin{align}  &l<l_{cr}:\quad \eta^*=0, \ c^*=0, \nonumber \\  &l>l_{cr}:\quad \eta^*=1, \ c^*=0, \label{eq52}
\end{align}

where $l_{cr}=\alpha l_{cr}^*=\frac{\alpha}{2}\sqrt{\frac{\beta}{2m}}$.

\vspace{5mm}

\item
For $d\geq2$,
\begin{align}  &l\leq l_{cr}:\quad I(l,f,\beta,d)=\beta \frac{l}{l_{cr}^*}, \nonumber \\ &l>l_{cr}:\quad I(l,f,\beta,d)<\beta(\alpha \wedge\frac{l}{l_{cr}^*}), \label{eq53}
\end{align}
and
\begin{align}  &l<l_{cr}:\quad \eta^*=0, \ c^*=0, \nonumber \\  &l>l_{cr}:\quad 0<\eta^*<1, \ 0<c^*<\sqrt{2\beta}, \label{eq54}
\end{align}

where $l_{cr}=\gamma_d l_{cr}^*$, with

\begin{equation} \gamma_d=\frac{-1+\sqrt{1+4m\alpha M_d^2}}{2m M_d^2}\in (0,\alpha), \label{eq55}
\end{equation}

where

\begin{equation} M_d=\frac{1}{2s_d m}\underset{R \in(0,\infty)}{\max}[g_d(R,0)-g_d(R,1)]. \label{eq56}
\end{equation}

Furthermore, $c^*>\sqrt{2\beta m}(1-\eta^*)$ when $l>l_{cr}$.

\end{enumerate}
\end{theorem}

Comparing this theorem to the corresponding one in \cite{E2003}, we see that in this case the extra factors $\alpha$ and $m$ appear ubiquitously in the formulas. Again, we note that the formulas above reduce to the corresponding ones in \cite{E2003} upon setting $\alpha=1$, $m=1$.

\section{Preparations}

In this section we prove five preparatory lemmas. Lemma~\ref{intensity} is needed directly in the proof of Theorem~\ref{teo2}. Lemma~\ref{overprod} and Lemma~\ref{comparison} are needed in the proof of Proposition~\ref{prop1}, which, in turn, is needed for proving Theorem~\ref{teo1}; the last two results
play  central roles in the proof of Theorem~\ref{teo2}. Lemma~\ref{decomp} is needed in the proof of both Theorem~\ref{teo2} and Theorem~\ref{teo1} (Lemma~\ref{hittingtimes} is needed too to prove the latter one). Finally, Lemma~\ref{hittingtimes} is also used in the proof of Proposition~\ref{prop1}.

\begin{lemma} \label{intensity}
Let $r>0$ and $b>0$. Then,
\begin{equation} \underset{t\rightarrow \infty}{\lim} \frac{1}{lt}\nu(B(bte,rt))=g_d(r,b). \nonumber
\end{equation}
\end{lemma}

\begin{proof} By substituting $y=x/t$, we see that
\begin{equation}
g_d(rt,bt)=\int_{B(0,rt)} \frac{\text{d}x}{|x+bte|^{d-1}}=t \int_{B(0,r)} \frac{\text{d}y}{|y+be|^{d-1}}=tg_d(r,b). \label{eq1}
\end{equation}
Recall that integrals of asymptotically equivalent positive continuous functions are asymptotically equivalent if one of the integrals diverges. Apply this result to the radial integral. By assumption, $\text{d}\nu/\text{d}x$ is continuous on $\mathbb{R}^d$. As the $d$-dimensional volume element has radial component $|x|^{d-1}\text{d}|x|$, the integrands of both radial integrals below are continuous on $\mathbb{R}_+$. Hence, we have
\begin{equation} \nu(B(bte,rt))=\int_{B(bte,rt)} \frac{\text{d}\nu}{\text{d}x}\text{d}x \sim
\int_{B(bte,rt)} \frac{l}{|x|^{d-1}}\text{d}x=l g_d(rt,bt). \label{eq2}
\end{equation}
Combining (\ref{eq1}) and (\ref{eq2}) yields $\underset{t\rightarrow \infty}{\lim} \frac{\nu(B(bte,rt))}{l t g_d(r,b)}=1$, which implies the result.
\end{proof}

\begin{lemma}[Overproduction] \label{overprod}
Let $\delta>0$. Then for any $t\geq 0$
\begin{equation} P(|Z(t)|>e^{(\beta m+\delta)t})\leq e^{-\delta t}. \nonumber
\end{equation}
\end{lemma}

\begin{proof} Use the fact that $E|Z(t)|=e^{m\beta t}$ (see \cite{KT1975}) and then apply Markov inequality.
\end{proof}

\begin{lemma}[Comparison] \label{comparison}
Suppose that $\lambda(0)=0$. Let $B\subset \mathbb{R}^d$ be open or closed. Let $x\in B$. Let $P_x$ be the law of Brownian motion, denoted by $W$, starting from $x$, and let $P_{\delta_x}$ be the law of BBM starting from $\delta_x$. Define the first exit times from $B$
\begin{align} \psi_B=&\inf\left\{t\geq 0:W(t)\in B^c\right\} \nonumber \\ \hat{\psi}_B=&\inf\left\{t\geq 0:|\text{supp}(Z(t))\cap B^c| \geq 1 \right\}. \nonumber
\end{align}
Then for any $k\in \mathbb{N}$ and $t\geq 0$
\begin{equation} P_{\delta_x}(\hat{\psi}_B>t||Z(t)|\leq k)\geq [P_x(\psi_B>t)]^k. \nonumber
\end{equation}
\end{lemma}

\begin{proof} See the corresponding proof in \cite{E2003}, which is written for strictly dyadic branching, and note that the proof can easily be extended to the case of general $\lambda$-GWP with $\lambda(0)=0$, since the population does not decrease for such processes, that is, for any $t\geq 0$, $|Z(s)|\leq |Z(t)|$ for all $s\in[0,t]$. 
\end{proof}

\begin{lemma}[Decomposition of the supercritical GWP] \label{decomp}
Let $N=(N(n))_{n\in \mathbb{N}}$ be a supercritical GWP with offspring distribution $\lambda$, where $\lambda(0)>0$, and let $f$ be the corresponding p.g.f. Let $q$ be the probability of extinction for $N$, and define $\bar{q}=1-q$. Let $N^*=(N^*(n))_{n\in \mathbb{N}}$ be the process, where $N^*(n)$ is the number of particles in generation $n$ that have an infinite line of descent.
\begin{enumerate}
	\item The law of $N$ given extinction is the same as that of a GWP with p.g.f.
	\begin{equation} \tilde{f}(s):=f(qs)/q. \label{ftilde}
	\end{equation}
	\item The law of $N^*$ given non-extinction is the same as that of a GWP with p.g.f.
	\begin{equation} f^*(s):=[f(q+\bar{q}s)-q]/\bar{q}. \label{fstar2}
	\end{equation}
	\item The law of $N$ given non-extinction is the same as that of a process $\xbar{N}$ generated as follows: To each particle of $N^*$ having $n_0$ children, add $n_1$ more children, that each start an independent GWP with p.g.f. $\tilde{f}$, where $n_1$ has the p.g.f.
	\begin{equation} f_{n_1}(s):=\frac{(D^{n_0}f)(qs)}{(D^{n_0}f)(q)},
	\end{equation}
	where $D$ is the differentiation operator, and all $n_1$ and all GWPs added are mutually independent given $N^*$. The resultant process is $\xbar{N}$.
	\item Let $X=(X(t))_{t\in \mathbb{R}^+}$ be a continuous time branching process with rate $\beta$ such that the number of particles in generation $n$ of $X$ is $N(n)$ for $n\in \mathbb{N}$. Let $X^*(t)$ be the number of particles with infinite line of descent present at time $t$. Then, the process $X^*=(X^*(t))_{t\in \mathbb{R}^+}$ given non-extinction is equal in law to a continuous time branching process with rate $\beta(1-f'(q))$ and offspring distribution, whose p.g.f. is
	\begin{equation} \bar{f}:=[f^*(s)-f'(q)s]/(1-f'(q)).  \label{fbar}
	\end{equation}
\end{enumerate}
\end{lemma}

\begin{remark}
The reader should note that in (\ref{fstar2}), positive mass on $1$ is allowed. For example, if $\lambda(0)=1/3$, $\lambda(2)=2/3$, $f(s)=1/3+(2/3)s^2$, then $q=1/2=\bar{q}$, and $f'(q)=2/3$, while $f^*(s)=(2/3)s+(1/3)s^2$. This distribution still has the same mean as originally (i.e., $4/3$), but here we put $2/3$ weight on $1$.

What (\ref{fbar}) says is that, equivalently (without putting mass on $1$), it has rate $\beta/3$ and p.g.f. $\bar{f}=s^2$.
\end{remark}

\begin{proof} See \cite[Proposition 4.10]{L1992} for a proof of parts 1, 2 and 3, and \cite[Section 1.12]{AN1972} for further details on the decomposition of supercritical GWPs. For a proof of part 4, observe that a continuous time branching process with rate $\beta$ and offspring distribution $\lambda$ is equal in law to a continuous time branching process with rate $\beta(1-\lambda(1))$ and offspring distribution $\bar{\lambda}$, where $\bar{\lambda}(1)=0$ and $\bar{\lambda}(j)=\lambda(j)/(1-\lambda(1))$ for $j\neq 1$. To complete the proof, see the proof of part 3 of \cite[Theorem 1]{OC2013}.
\end{proof}

We note that the particles of infinite line of descent, which we refer to as the `skeleton particles' in this work, have been referred to as `prolific individuals' in the past in \cite{BFM2008}. Furthermore, the term `backbone' has also been used for the analogue of the term `skeleton' in the context of continuous-state branching processes and superprocesses, see e.g. \cite{BKM2011,KMP2012,KR2012}.

\begin{lemma}[Brownian hitting times in $\mathbb{R}^d$] \label{hittingtimes}
Let $B=(B_t)_{t\geq 0}$ be a Brownian motion in $\mathbb{R}^d$ starting at the origin, with corresponding probability $P^{(d)}_0$, and let $k>0$. Then, as $t\rightarrow \infty$,
\begin{equation} P^{(d)}_0\left(\underset{0\leq s \leq t}{\sup}|B_s|\geq kt\right)=\exp\left\{-\frac{k^2 t}{2}(1+o(1))\right\}. \label{eq3}
\end{equation}
Furthermore, for $d=1$, even the following stronger statement is true. Let $m_t=\inf_{0\leq s \leq t}B_s$ and $M_t=\sup_{0\leq s \leq t}B_s$. (The process $M_t-m_t$ is the \text{range process} of $B$.) Then, as $t\rightarrow \infty$,
\begin{equation} P^{(1)}_0(M_t-m_t\geq kt)=\exp\left\{-\frac{k^2 t}{2}(1+o(1))\right\}. \label{eq4}
\end{equation}
\end{lemma}

\begin{proof} For $d=1$, recall that \cite[equation 7.3.3]{KT1975} for any $a>0$,
\begin{equation} P^{(1)}_0\left(\underset{0\leq s \leq t}{\sup}|B_s|\geq a\right)=\sqrt{\frac{2}{\pi}}\int^{\infty}_{a/\sqrt{t}}\exp(-y^2/2)\text{d}y.  \nonumber
\end{equation}
Setting $a=kt$ above, we see via l'H\^{o}pital's rule that
\begin{equation} P^{(1)}_0\left(\underset{0\leq s \leq t}{\sup}|B_s|\geq kt\right)=\exp\left\{-\frac{k^2 t}{2}(1+o(1))\right\}. \nonumber
\end{equation}
We now extend this result to $d\geq 2$. Note that the lower bound holds trivially since the projection of a d-dimensional Brownian motion onto the first coordinate axis is a one-dimensional Brownian motion. In order to prove the upper bound, we first prove (\ref{eq4}).

Let $d=1$. Define
\[\theta_c=\inf\left\{s\geq 0|M_s-m_s=c\right\}.
\]
According to \cite[p. 199]{CY2003} and references therein, the Laplace transform of $\theta_c$ satisfies
\begin{equation} E\exp\left(-\frac{\lambda^2}{2}\theta_c\right)=\frac{2}{1+\cosh(\lambda c)},\quad \lambda>0. \nonumber
\end{equation}
Hence, by exponential Markov inequality,
\begin{equation} P^{(1)}_0(\theta_c \leq t)\leq \exp\left(\frac{\lambda^2}{2}t\right)E\exp\left(-\frac{\lambda^2}{2}\theta_c\right)=\exp\left(\frac{\lambda^2}{2}t\right)\frac{2}{1+\cosh(\lambda c)}. \nonumber
\end{equation}
Taking $c=kt$, one obtains
\begin{equation} P^{(1)}_0(M_t-m_t\geq kt)=P^{(1)}_0(\theta_{kt} \leq t)\leq \exp\left(\frac{\lambda^2}{2}t\right)\frac{2}{1+\cosh(\lambda kt)}. \nonumber
\end{equation}
Optimizing over $\lambda$, we see that the estimate is the sharpest when $\lambda=k$, in which case, we arrive at the desired upper bound. Since 
\[ P^{(1)}_0\left(\sup_{0\leq s \leq t}|B_s|\geq kt\right)\leq P^{(1)}_0(M_t-m_t\geq kt), \] 
the lower bound coincides with the upper bound, which implies (\ref{eq4}).

For $d\geq 2$, let us consider more generally the Brownian motion started at $x$, denoted by $B^x$, $x\in\mathbb{R}^d$, with law $P^{(d)}_x$, and let $r=|x|$. We will denote $P^{(d)}_0$ simply by $P^{(d)}$, that is, $P^{(d)}$ denotes $d$-dimensional Wiener measure. Then the process $R^r=(R^r_t)_{t\geq 0}$ with state space $[0,\infty)$, defined by $R^r=|B^x|$  is called the \textit{$d$-dimensional Bessel process}, provided  $r>0$. We will use the well known fact \cite[Example 8.4.1]{O2003} that the process 
\begin{equation} R^r_t-r-\int_0^t \frac{d-1}{2R^r_s}\text{d}s \label{bessel}
\end{equation}
is a standard one-dimensional Brownian motion for each $r>0$; its law is thus $P^{(1)}$. Since $R^r$ is the strong solution to a stochastic differential equation (one gets an $\omega$-wise solution by making the expression in \eqref{bessel} equal to a standard Brownian motion), the probability distribution $P^{(1)}(R^r\in \cdot)$ is well defined.

Using the strong Markov property of Brownian motion, applied at $\tau$, the first hitting time of the $\rho$-sphere by $B$, it is clear that in order to verify the lemma, it suffices to prove it when $B$ is replaced by $B^x$, with $|x|=\rho>0$. This follows, since for every fixed $\rho$ such that $0<\rho<kt$, we have
\begin{equation}  P^{(d)}\left(\underset{0\leq s \leq t}{\sup}|B_s|\geq kt\right)\leq  P^{(d)}\left(\underset{0\leq s \leq t+\tau}{\sup}|B_s|\geq kt\right)=P^{(1)}\left(\underset{0\leq s \leq t}{\sup}R^{\rho}_s\geq kt\right). \label{eq5}
\end{equation} 

Define the sequence of stopping times $0=:\tau_0<\sigma_0<\tau_1<\sigma_1<\cdots$ with respect to the filtration generated by $R^{\rho}$ as follows:
\begin{equation} \quad \sigma_0:=\inf\left\{s>0:R^{\rho}_s=\rho/2\right\}, \nonumber
\end{equation}
and for $i\geq 1$,
\begin{equation} \tau_{i+1}:=\inf\left\{s>\sigma_i:R^{\rho}_s=\rho\right\}, \quad \sigma_{i+1}:=\inf\left\{s>\tau_{i+1}:R^{\rho}_s=\rho/2\right\}. \nonumber
\end{equation}
Let $W^{(\rho)}_t:=R^{\rho}_t-\rho-\int_0^t \frac{d-1}{2R^{\rho}_{s}}\text{d}s$  (note that $W^{(\rho)}_0:=0$). Then, for $i\geq 0$ and $s\in [\tau_i,\sigma_i]$,
\begin{equation} R^{\rho}_s=\rho+\int_{\tau_i}^s \frac{d-1}{2R^{\rho}_z}\text{d}z+W^{(\rho)}_s-W^{(\rho)}_{\tau_i}\leq \rho+\rho^{-1}(d-1)\Delta_i^s+W^{(\rho)}_s-W^{(\rho)}_{\tau_i}, \label{eq6}
\end{equation}
where  $\Delta_i^s:=s-\tau_i$. 

Since $R_s\leq \rho$ for $\sigma_i \leq s \leq \tau_{i+1}$ and $i\geq 0$, it is also clear that for $t>\rho/k$, the relation $\sup_{0\leq s \leq t}R^{\rho}_s\geq kt$ is tantamount to
\begin{equation} \underset{i\geq 0}{\sup} \quad \underset{\tau_i\wedge t<s\leq \sigma_i\wedge t}{\sup}R^{\rho}_s\geq kt.  \nonumber
\end{equation}
Putting this together with (\ref{eq6}), it follows that $P^{(1)}(\sup_{0\leq s \leq t}R^{\rho}_s \geq kt)$ can be estimated from above by 
\begin{equation} P^{(1)}\left(\exists \ i \geq 0, \exists \ \tau_i \wedge t<s\leq \sigma_i \wedge t:W^{(\rho)}_s-W^{(\rho)}_{\tau_i}\geq kt-\rho^{-1}(d-1)\Delta_i^s-\rho \right). \nonumber
\end{equation}
This can be further estimated from above by
\begin{equation} P^{(1)}\left(M_t-m_t\geq \left[k-\frac{d-1}{\rho}-\delta\right]t \right), \nonumber
\end{equation} 
for any $\delta>0$, as long as $t\geq \rho/\delta$. To complete the proof, fix $\rho$, $\delta>0$, let $t\rightarrow \infty$, and use the already proved relation (\ref{eq4}), along with (\ref{eq5}). Finally let $\delta\rightarrow 0$ and $\rho\rightarrow \infty$.

\end{proof}

We note that precise estimates for the tail of $P^{(d)}_0\left(\underset{0\leq s \leq t}{\sup}|B_s|\geq kt\right)$ in Lemma~\ref{hittingtimes} are given in \cite[Theorem 2.1]{G1996}. Nonetheless, the estimate (\ref{eq3}), for which we have provided a self-contained proof, suffices for our purposes.

\section{The speed of BBM}
In this section, we extend the well-known convergence in probability result of McKean \cite{MK1975,MK1976} regarding the speed of BBM, which holds when the dimension is one, and when each particle is certain to give at least one offspring, i.e., $\lambda(0)=0$. We first extend McKean's result to any dimension, and finally to the case where $\lambda(0)>0$. The results of this section are central in proving Theorem~\ref{teo2}.

Recall from (\ref{range}) that $\cup_{s\in[0,t]} \text{supp}(Z(s))$ denotes the range of $Z$ up to time $t$, and define
\begin{equation} M(t)=\inf \left\{r>0:R(t)\subseteq B(0,r)\right\} \quad \text{for} \quad d\geq 1, \label{M}
\end{equation}
to be the radius of the minimal ball containing $R(t)$. We define the speed at time $t$ of a BBM as $M(t)/t$. Note that this is slightly different from the classical notion of speed (see for example \cite{MK1975,CR1988,K2005}), which is $X(t)/t$, where $X(t):=\inf \left\{r> 0:\text{supp}(Z(t))\subseteq B(0,r)\right\}$ is the radius of the minimal ball containing the support of the BBM at time $t$ only. Although the two definitions of speed above are different, it is easy to prove that they should have the same asymptotics as $t\rightarrow\infty$.

\begin{proposition}[Speed of BBM] \label{prop1}
Suppose that $\lambda(0)=0$.
For $d\geq 1$, $\frac{M(t)}{t}$ converges to $\sqrt{2\beta m}$ in $P$-probability as $t\rightarrow \infty$.
\end{proposition}

\begin{remark}
It is easy to see that the quantity $\beta m$ is invariant under changing the setting from the `canonical one' (that is, putting zero mass on $1$) to a non-canonical one (that is, assigning positive mass to $1$).
\end{remark}

\begin{proof} For $d=1$, the reader is referred to \cite{MK1975,MK1976} and \cite{CR1988}. Now let $d\geq 2$. Since the projection of $Z$ onto the $1$st coordinate axis is a one-dimensional BBM with branching rate $\beta$, the lower estimate for convergence in probability follows from the result for $d=1$ and the inequality

\begin{equation} P\left(\frac{M(t)}{t}>\sqrt{2\beta m}-\varepsilon\right)\geq P^*\left(\frac{M(t)}{t}>\sqrt{2\beta m}-\varepsilon\right)   \nonumber
\end{equation}
$\forall \: \varepsilon>0$ and $\forall \: t$, where $P^*$ denotes the law of the one-dimensional projection of $Z$.

To prove the upper estimate for convergence in probability, let $\varepsilon>0$ and let $B=B(0,(\sqrt{2\beta m}+\varepsilon)t)$. Pick $\delta>0$ such that

\begin{equation} \frac{1}{2}(\sqrt{2\beta m}+\varepsilon)^2>\beta m+\delta. \label{eq7}
\end{equation}
Recall that $\psi_B$ and $\hat{\psi}_B$ are the first exit times from $B$ for a single Brownian particle starting at the origin and for a BBM, respectively. See that these events are identical: $\left\{\frac{M(t)}{t}>\sqrt{2\beta m}+\varepsilon\right\}=\left\{\hat{\psi}_B\leq t\right\}$.
Estimate

\begin{align} P\left(\frac{M(t)}{t}>\sqrt{2\beta m}+\varepsilon\right)\leq & \ P(|Z(t)|>\lfloor e^{(\beta m+\delta)t} \rfloor) \nonumber \\ &+P(\hat{\psi}_B\leq t|\  |Z(t)|\leq \lfloor e^{(\beta m+\delta)t} \rfloor). \label{eq8}
\end{align}
By Lemma~\ref{overprod}, the first term on the right-hand side tends to zero exponentially fast. Now consider the second term. Recall that $P_x$ is the law of Brownian motion starting from $x$, and $P_{\delta_x}$ is the law of BBM starting from $\delta_x$. By Lemma~\ref{comparison}, we have the estimate

\begin{equation} P_{\delta_0}(\hat{\psi}_B> t||Z(t)|\leq \lfloor e^{(\beta m+\delta)t} \rfloor)\geq [P_0(\psi_B>t)]^{\lfloor e^{(\beta m+\delta)t} \rfloor} . \nonumber
\end{equation}
By Lemma~\ref{hittingtimes},

\begin{equation} [P_0(\psi_B>t)]^{\lfloor e^{(\beta m+\delta)t} \rfloor}=\left[1-\exp\left(-\frac{[(\sqrt{2\beta m}+\varepsilon)t]^2}{2t}\right)\exp(o(t))\right]^{\lfloor e^{(\beta m+\delta)t} \rfloor}. \label{eq9}
\end{equation}
By (\ref{eq7}), using the binomial expansion, the right-hand side of (\ref{eq9}) tends to $1$ exponentially fast as $t\rightarrow \infty$, so that (\ref{eq8}) yields the desired upper estimate

\begin{equation} P\left(\frac{M(t)}{t}>\sqrt{2\beta m}+\varepsilon\right)\rightarrow 0, \nonumber
\end{equation}
which completes the proof.
\end{proof}
{\bf Note:} For the upper estimate, one can give another, equally simple proof, by exploiting the so-called Many-to-One Lemma (Lemma 1.1.6 in \cite{E2015}).

\medskip
Next, we turn our attention to the speed, when conditioned on non-extinction.

\begin{theorem}[Conditional speed of BBM] \label{teo1}
Suppose that the underlying GWP is supercritical, i.e., $m>0$. Then, for $d\geq 1$, conditioned on non-extinction, $\frac{M(t)}{t}$ converges to $\sqrt{2\beta m}$ in $P$-probability as $t\rightarrow \infty$.
\end{theorem}

Regarding the speed of BBM, we refer the reader to \cite{MK1975,MK1976} for the case $d=1$, and where $\lambda(0)=0$. We note that almost sure speed results exist too (see e.g. \cite{K2005}), but for our purposes, convergence in probability suffices. For $d\geq 2$ and $\lambda(0)=0$, we have derived the speed in Proposition~\ref{prop1} above. Theorem~\ref{teo1}, which will be proved below, covers the general case where $\lambda(0)$ can be nonzero. In Theorem~\ref{teo1}, we see that conditioned on non-extinction of the underlying GWP, the speed of BBM remains as $\sqrt{2\beta m}$ compared to a BBM with the same offspring mean but with $\lambda(0)=0$, which can be explained as follows. When the supercritical GWP is conditioned on non-extinction, and decomposed into skeleton and doomed particles (see Lemma~\ref{decomp}), the doomed particles do not contribute to the speed, and the skeleton particles alone form another GWP with the same mean as the original GWP.
\subsection{Proof of Theorem~\ref{teo1}}

\noindent\underline{Lower bound:}

We emphasize that in this proof, the branching rate for all types of particles for all processes considered is constant and is equal to $\beta$, and all probabilities written should be understood as conditioned on the non-extinction of the underlying GWP.

First, decompose the supercritical GWP into two sets of particles: particles with infinite line of descent, which, following the terminology in \cite{BSS2013}, we call skeleton particles; and particles with finite line of descent, which we call the doomed particles. By Lemma~\ref{decomp}, each skeleton particle gives rise to a GWP consisting only of skeleton particles with p.g.f.
\begin{equation} f^*(s):=[f(q+\bar{q}s)-q]/\bar{q}, \label{eq10}
\end{equation}
where $q$ is the probability of extinction for the underlying GWP and $\bar{q}=1-q$. Also, by Lemma~\ref{decomp}, each doomed particle starts her own single-type GWP with p.g.f.
\begin{equation} \tilde{f}(s):=f(qs)/q. \label{eq11}
\end{equation}
Observe that $(f^*)'(1)=f'(1)=\mu$. Recall that $m=\mu-1>0$. It is clear that each skeleton particle produces at least one skeleton particle, and can produce doomed particles as well, whereas a doomed particle only produces doomed particles.

Now we translate this decomposition into the language of multi-type GWPs, see \cite[Chapter 5]{AN1972}. We define the skeleton particles to be of type $1$, and the doomed particles to be of type $2$, so that the underlying GWP is a two-type GWP, and we have the following decomposition for the BBM:
\begin{equation} (Z(t))_{t\geq 0}=(Z^1(t),Z^2(t))_{t\geq 0}, \nonumber
\end{equation}
where $Z^1$ is the process consisting of the skeleton particles, and $Z^2$ is the one consisting of the doomed particles. Note that $Z^1$ by itself is a BBM, however, $Z^2$ is not. Let $|Z^j(t)|$, $j=1,2$ be the number of particles of type $j$ existing at time $t$. Observe that conditioning $Z$ on non-extinction is equivalent to the initial condition $(|Z^1(0)|,|Z^2(0)|)=(1,0)$, which says that the process starts with one skeleton particle, hence never dies out. For $d\geq1$, let us define the range of $Z^1$ as $R^1$, and the range of $Z^2$ as $R^2$. Similarly, define
\begin{align} M^1(t):=&\inf \left\{r>0:R^1(t)\subseteq B(0,r)\right\}, \nonumber \\
              M^2(t):=&\inf \left\{r>0:R^2(t)\subseteq B(0,r)\right\}. \nonumber
\end{align}
Now since $M(t)=\max\left\{M^1(t),M^2(t)\right\}$, $f^*(0)=0$, $(f^*)'(1)-1=m$, and $\beta m$ in the speed is invariant under how one describes the branching (see the remark after Proposition~\ref{prop1}), Proposition~\ref{prop1} immediately gives the desired lower bound: $\forall \: \varepsilon>0$,
\begin{equation} P\left(\frac{M(t)}{t}>\sqrt{2\beta m}-\varepsilon\right)\rightarrow 1 \quad \text{as} \quad t\rightarrow \infty. \nonumber
\end{equation}

\noindent\underline{Upper bound:}

Recall that $\mathcal{E}$ denotes extinction and $P(\mathcal{E}^c)>0$.
Clearly, $\forall \: \varepsilon>0$,
\begin{align*} P\left(\frac{M(t)}{t}>\sqrt{2\beta m}+\varepsilon\right)&=P\left(\frac{M(t)}{t}>\sqrt{2\beta m}+\varepsilon\mid \mathcal{E}^c\right)P(\mathcal{E}^c)\\
&+P\left(\frac{M(t)}{t}>\sqrt{2\beta m}+\varepsilon;\ \mathcal{E}\right).
\end{align*}
The last term converges to zero, by bounded convergence. Consequently,  the limit
$$\lim_{t\to\infty}P\left(\frac{M(t)}{t}>\sqrt{2\beta m}+\varepsilon\mid \mathcal{E}^c\right)=0$$
is tantamount to
\begin{equation*}
\lim_{t\to\infty}P\left(\frac{M(t)}{t}>\sqrt{2\beta m}+\varepsilon\right)=0.
\end{equation*}
When $\lambda(0)=0$, then of course,  we are immediately done by Proposition \ref{prop1}.

In the general case, we finish the proof as follows. Let $\mathcal{N}_t^*$ denote the {\it total progeny} up to $t$, that is, all the particles that have been born not later than $t$ (but possibly not present at $t$). Then, using the  union bound, for $\gamma>0$,
\begin{align}\label{must.go.to.zero} P\left(M(t)>\gamma t\right)
& = P\left(\exists u\in \mathcal{N}_t^*\ :\ \sup_{0\le s\le t}|X_u(s)|>\gamma t\right) \nonumber \\
& \le E(|\mathcal{N}_t^*|)\mathbf{P}_0\left(\sup_{0\le s\le t}|B(s)|>\gamma t\right),
\end{align}
where $B$ is a generic standard Brownian particle starting at the origin, with probability $\mathbf P_0$. Now pick $\gamma=\sqrt{2\beta m}+\varepsilon$. By Lemma \ref{hittingtimes}, the following estimate is sufficient to complete the proof of the upper bound: for any $c>m\beta$,
\begin{equation}
g(t):=E(|\mathcal{N}_t^*|)\le  \exp(c t)\label{homework}
\end{equation}
holds for all sufficiently large $t$. Indeed, then picking $c\in (\beta m,\gamma^2/2)$, the righthand side of \eqref{must.go.to.zero} converges to zero. 

To prove \eqref{homework}, condition on the first branching time, and obtain the functional equation for $g$:
$$g(t)=e^{-\beta t}+\int_0^t [p_0+\mu g(t-s)]\beta e^{-\beta s}\, \mathrm{d}s=e^{-\beta t}+\int_0^t [p_0+\mu g(s)]\beta e^{-\beta (t-s)}\, \mathrm{d}s.$$
(We include the $p_0$ term as we  describe the total progeny up to $t$.)
That is,
$$G(t):=e^{\beta t}g(t)=1+\int_0^t [p_0+\mu g(s)]\beta e^{\beta s}\, \mathrm{d}s.$$
Differentiating,
$$G'(t)=[p_0+\mu g(t)]\beta e^{\beta t}=p_0 \beta e^{\beta t}+\mu\beta G(t).$$
Comparing $E(|\mathcal{N}_t^*|)$ with just the expected population size at $t$, obviously, $g(t)\ge e^{m\beta t}$, that is, $G(t)\ge e^{\mu\beta t}$.
Since $\mu>1$, for any given $\delta>0$,
$$G'(t)\le (\delta+\mu\beta) G(t),$$
for $t\ge T_{\delta}:=\frac{\log(p_0\beta/\delta)}{m\beta}$. By Gronwall's inequality, $G(t)\le G(T_{\delta}) \cdot e^{(\delta+\mu\beta)t}$ and so 
$g(t)\le G(T_{\delta}) \cdot e^{(\delta+m\beta)t}$ for $t\ge T_{\delta}$.
 \qed

\begin{remark} [Speed comparison]\label{pgfcomparison}
Using Jensen's inequality, it is easy to show that if $s\in[0,1]$, then $f^*(s) \leq \tilde{f}(s)$. Exploiting this inequality, and using the notation of the lower bound in the previous proof,  one can show then that
for any Borel set $B$, and any fixed time $t\geq s$,
\begin{equation} P^1_{s,x}(R^1(t)\subset B)\leq P^2_{s,x}(R^2(t)\subset B), \label{remain}
\end{equation}
where  $P^1_{s,x}$ and $P^2_{s,x}$ are the corresponding probabilities for processes that start at time $s$ with a single particle at position $x\in \mathbb{R}^d$.
\end{remark}

We close this section with the following proposition, which will be used in the proof of the lower bound of Theorem~\ref{teo2}. Although the statement is not about the speed of the process, we included it here, as some machinery developed in this section is being used.

\begin{proposition}\label{twotype}
Let $N$ be a supercritical Galton-Watson process with offspring p.g.f.\ $f$ and extinction probability $0<q<1$, and let $N^*$ be the process generated by the skeleton particles. Let $\bar{q}=1-q$. Then, the two-type Galton-Watson process $\mathbf{N}=(N^*,N-N^*)$, where the particles of type 1 and type 2 are composed of skeleton and doomed particles, respectively, has the mean matrix
\begin{equation} M=
\begin{pmatrix} 
\mu & (\mu-\kappa)\frac{q}{\bar{q}}  \\  0 & \kappa
\end{pmatrix}, \label{mm}
\end{equation}
where $\mu:=f'(1)$, $\kappa:=f'(q)$, and $M_{ij}=E\mathbf{N}_j^i(1)$, $i,j=1,2$, are the expected number of offspring of type $j$ that a single particle of type $i$ has.
\end{proposition}

\begin{proof} 
Let $f^1$, $f^2$ be the offspring p.g.f.s for skeleton and doomed particles, respectively. The p.g.f.\ for doomed particles was computed in part (1) of Lemma~\ref{decomp}: $f^2(s,t)=f(qt)/q$. If we let $\xi$, $\xi^*$ be random variables that are independent copies of $N(1)$, $N^*(1)$, respectively, then the p.g.f.\  for skeleton particles is the same as the joint p.g.f.\ of $\xi-\xi^*$ and $\xi^*$ conditioned on the event $\left\{\xi^*\geq 1\right\}$. We have
\begin{align} f^1(s,t)=E[s^{\xi^*}t^{\xi-\xi^*}\mid \xi^*\geq 1]=& \frac{E[s^{\xi^*}t^{\xi-\xi^*}\mathbbm{1}_{\left\{\xi^*\geq 1\right\}}]}{P(\xi^*\geq 1)} \nonumber \\
=& \frac{E[s^{\xi^*}t^{\xi-\xi^*}]-E[s^{\xi^*}t^{\xi-\xi^*}\mathbbm{1}_{\left\{\xi^*=0\right\}}]}{\bar{q}} \nonumber \\
=& \frac{f(\bar{q}s+qt)-f(qt)}{\bar{q}}, \label{eq227}
\end{align}
where we have used that $E[s^{\xi^*}t^{\xi-\xi^*}]=f(\bar{q}s+qt)$ (see \cite[Proposition 4.10]{L1992}), and part (1) of Lemma~\ref{decomp} in the last equality.

The mean matrix $M$ for a $k$-type GWP is defined to be the $k\times k$ matrix whose $(i,j)$th entry is $M_{ij}=\frac{\partial f^i (\mathbf{s})}{\partial s_j}\biggr\vert_{\mathbf{s}=1}$, where $\mathbf{s}=(s_1,\ldots,s_k)\in[0,1]^k$, and $f^i (\mathbf{s})$ is the offspring p.g.f.\ for particles of type $i$.

As we have a two-type process, let us set $\mathbf{s}=(s,t)$. Setting $s=1$ in (\ref{eq227}) yields the p.g.f.\ for the number of doomed offspring of a skeleton particle:
\begin{equation}  E[t^{\xi-\xi^*}|\xi^*\geq 1]=\frac{f(qt+\bar{q})-f(qt)}{\bar{q}}. \label{joint3}
\end{equation}
Differentiating (\ref{joint3}) with respect to $t$ and setting $t=1$ gives the desired expectation 
\begin{equation}  \frac{\partial f^1 (\mathbf{s})}{\partial t}\biggr\vert_{\mathbf{s}=1}=(f'(1)-f'(q))\frac{q}{\bar{q}}. \label{joint4}
\end{equation} 
It is clear from parts (1) and (2) of Lemma~\ref{decomp} that 
\begin{align} \frac{\partial f^1 (\mathbf{s})}{\partial s}\biggr\vert_{\mathbf{s}=1}=&\ f'(1), \nonumber \\
\frac{\partial f^2 (\mathbf{s})}{\partial s}\biggr\vert_{\mathbf{s}=1}=&\ 0, \nonumber \\
\frac{\partial f^2 (\mathbf{s})}{\partial t}\biggr\vert_{\mathbf{s}=1}=&\ f'(q). \label{joint5}
\end{align}
This completes the proof.
\end{proof}

\section{Proof of Theorem~\ref{teo2}}
Part 3 of the theorem follows since $\tau<\infty$ almost surely for a non-supercritical process. For a detailed proof of part 3, see \cite[Theorem 1]{OC2013}. Since $\lambda(0)=0$ gives $\alpha=1$ and $P(\mathcal{E}^c)=1$, conditioning the process on $\mathcal{E}^c$ is the same as not conditioning it when $\lambda(0)=0$. Hence, we may extend part 2 to cover for the case $\lambda(0)=0$, and prove the first two parts of the theorem together. Note that the key ingredients in the proof are results concerning the decomposition of supercritical GWPs and the speed of BBM, namely Lemma~\ref{decomp} and Theorem~\ref{teo1}.

\subsection{Proof of the lower bound}
Let $m>0$. Fix $\beta$, $d$, $f$ and $\eta$, $c$. The scenario below yields the desired lower bound for the annealed survival probability conditioned on non-extinction, i.e., $(\mathbb{E}\times P) \left(T>t|\mathcal{E}^c\right)$.

\begin{enumerate}
\item Recall that the first particle must be of infinite line of descent once the process is conditioned on non-extinction. Suppress the branching of the initial skeleton particle  so that there is precisely one skeleton particle in the system up to time $\eta t$. This is equivalent to requiring that whenever the skeleton particle branches within the period $[0,\eta t]$, it gives precisely one offspring of its own kind (and possibly some doomed offspring). Therefore, by the fourth part of Lemma~\ref{decomp}, this event has probability
\[ \exp[-\beta \alpha \eta t]. \]

\item For $d\geq 2$, empty the two-sided cylinder (``tube") $T_t$, defined by
\begin{equation} 
\begin{split}
T_t=\left\{ \vphantom{\sqrt{x_2^2+\ldots+x_d^2}\leq r(t)+h(t)} x=(x_1,\ldots, x_d)\in \mathbb{R}^d:|x_1|\leq kt+h(t),\right. \\ \left. \ \ \ \sqrt{x_2^2+\ldots+x_d^2}\leq r(t)+h(t)\right\},  \nonumber
\end{split}
\end{equation}
where $k>c$, and $t\mapsto r(t)$ and $t\mapsto h(t)$ are non-negative mappings that are picked such that $\lim_{t\rightarrow \infty}r(t)/t=0$, $\lim_{t\rightarrow \infty}r(t)=\infty$ and $\lim_{t\rightarrow \infty}h(t)/t=0$, $\lim_{t\rightarrow \infty}h(t)=\infty$. By `emptying a region', we mean clearing a region from trap points.
By the upper bounds above on $r(t)$ and $h(t)$, Lemma~\ref{intensity} gives that the probability to empty $T_t$ is
$\exp[o(t)]$
for $d\geq 2$. For $d=1$, empty the line $T^1_t:=\left\{x_1\in \mathbb{R}:|x_1|\leq 2h(t) \right\}$, which again costs a probability of
$\exp[o(t)].$

\item For $d\geq 2$, move the single skeleton particle (see the remark after Lemma~\ref{decomp}) during the period $[0,\eta t]$ so that it is at a specific site at distance $ct+o(t)$ from the origin at time $\eta t$ (assume without loss of generality that the specific site has first coordinate $ct+o(t)$ and all other coordinates zero), and confine it to the smaller tube $\hat{T}_t$, defined by
\begin{equation} \hat{T}_t=\left\{x=(x_1,\ldots, x_d)\in \mathbb{R}^d:|x_1|\leq kt,\ \sqrt{x_2^2+\ldots+x_d^2}\leq r(t)\right\}, \nonumber
\end{equation}
up to time $\eta t$.

Let $P_0$ be the $d$-dimensional Wiener measure, $W^1$ be the projection of the $d$-dimensional Brownian motion onto the first coordinate axis, and decompose the Brownian motion into an independent sum $W=W^1+W^{d-1}$, where $W^{d-1}$ is hence implicitly defined. Define the events
\begin{align} A_t=& \left\{ct\leq |W^1_{\eta t}| \leq ct+o(t)\right\},  \nonumber \\
              B_t=& \left\{|W^1_s|\leq kt \ \  \forall \  0\leq s \leq \eta t \right\},  \nonumber \\
							C_t=& \left\{|W^{d-1}_s|\leq r(t) \ \  \forall \  0\leq s \leq \eta t \right\}. \nonumber
\end{align}
Note that the event $A_t\cap B_t\cap C_t$ is exactly the desired scenario for this part. The formula for the transition density for $d$-dimensional Brownian motion gives $P_0(A_t)=\exp[-\frac{c^2}{2\eta}t+o(t)]$, and Lemma~\ref{hittingtimes} gives $P_0(B_t^c)=\exp[-\frac{k^2}{2\eta}t+o(t)]$, so it follows that
\begin{align} \exp\left[-\frac{c^2}{2\eta}t+o(t)\right]=&P_0(A_t)\geq P_0(A_t\cap B_t)\geq P_0(A_t)-P_0(B_t^c) \nonumber \\
              =&\exp\left[-\frac{c^2}{2\eta}t+o(t)\right]-\exp\left[-\frac{k^2}{2\eta}t+o(t)\right] \nonumber \\
							=&\exp\left[-\frac{c^2}{2\eta}t+o(t)\right], \nonumber
\end{align}
which yields
\begin{equation} P_0(A_t\cap B_t)=\exp\left[-\frac{c^2}{2\eta}t+o(t)\right]. \label{eq20}
\end{equation}
Since $r(t)\rightarrow \infty$, we have $P_0(C_t)=\exp[o(t)]$. Combining this with (\ref{eq20}) and using the independence of $W^1$ and $W^{d-1}$, we find
\begin{equation} P_0(A_t\cap B_t \cap C_t)=\exp\left[-\frac{c^2}{2\eta}t+o(t)\right]. \nonumber
\end{equation}

For $d=1$, note that the function in (\ref{I}) is minimized when $c=0$ (see the proof of Theorem~\ref{teo3}). Confine the single skeleton particle to the one-dimensional tube $\hat{T}^1_t:=\left\{x_1\in \mathbb{R}:|x_1|\leq h(t) \right\}$ up to time $\eta t$. Since the length of the tube tends to infinity as $t$ tends to infinity, the probability of this event is
$\exp[o(t)].$

\item Confine all the doomed particles that are created within the period $[0,\eta t]$ to the larger tubes $T_t$ and $T^1_t$ for $d\geq 2$ and $d=1$, respectively, up to time $t$. Since the number of occurrences of branching up to time $t$ along a single ancestral line is a Poisson process with mean $\beta t$, and a skeleton particle on average produces $(\mu-\kappa)q/(1-q)$ doomed particles every time it branches (see Proposition~\ref{twotype}), it follows that at most $\lfloor \beta \eta t (\mu-\kappa)q/(1-q) \rfloor$ doomed particles are produced along the single skeletal line up to time $\eta t$ with probability $\exp[o(t)]$. Let the radii\footnote{By radius we mean the radius of the smallest ball centered at the root of the subtree, containing the (finite) range of the subtree.} of the subtrees initiated by these $\lfloor \beta \eta t (\mu-\kappa)q/(1-q) \rfloor=:n(t)$ doomed particles be $\rho_1, \rho_2, \ldots, \rho_{n(t)}$. If $K(t):=\max\left\{\rho_1, \rho_2, \ldots, \rho_{n(t)} \right\}$, then
\begin{equation}
P(K(t)<h(t))=P(\rho_1<h(t))^{n(t)}. \label{cost}
\end{equation}
Now, since a doomed subtree is almost surely finite, $\lim_{t\rightarrow \infty}h(t)=\infty$ implies that $P(\rho_1<h(t))\rightarrow 1$, which in turn implies that the probability in (\ref{cost}) is $\exp[o(t)]$. To see this, note that $n(t)$ is in the form $Ct$, $C>0$, and
the right-hand side of (\ref{cost}) is $\exp[n(t)\log P(\rho_1<h(t))]$ where $\log(P(\rho_1<h(t))\rightarrow 0$ as $t\rightarrow \infty$. We have already confined the single skeleton particle to the smaller tube $\hat{T}_t$ for $d\geq 2$ ($\hat{T}^1_t$ for $d=1$). Since each doomed particle that is produced along the single skeletal line within the period $[0,\eta t]$ must be produced within the smaller tube $\hat{T}_t$($\hat{T}^1_t$), and the dimensions of the larger tube $T_t$($T^1_t$) all exceed the corresponding dimensions of the smaller tube by $h(t)$, we conclude that all of the doomed particles that are created within the period $[0,\eta t]$ stay inside the larger tube with probability
$\exp[o(t)].$

\item Empty a $(\sqrt{2\beta m}+\varepsilon)(1-\eta)t$-ball around the position of the single skeleton particle at time $\eta t$. By Lemma~\ref{intensity}, this event has probability
\[ \exp\left[-l g_d((\sqrt{2\beta m}+\varepsilon)(1-\eta),c) t+o(t)\right].  \]

\item Require the branching system initiated by the single skeleton particle present at time $\eta t$ to stay inside the $(\sqrt{2\beta m}+\varepsilon)(1-\eta)t$-ball during the remaining time $(1-\eta)t$, where $\varepsilon>0$. The probability of this event conditioned on non-extinction is $\exp[o(t)].$ To see this, note that when $m>0$, by Theorem~\ref{teo1}, for every $\varepsilon>0$, $P\left(M(t)> (\sqrt{2\beta m}+\varepsilon) t\ |\ \mathcal{E}^c\right) \rightarrow 0$ as $t\rightarrow \infty$.

\end{enumerate}

Since the motion, branching and trap formation mechanisms are independent of each other, minimizing the cost of all these events over the parameters $\eta$ and $c$, and letting $\varepsilon\rightarrow 0$ in part 5 provides us with the desired lower estimate for $(\mathbb{E}\times P) (T>t \: |\: \mathcal{E}^c)$. Note that the survival scenario described above is composed of three large deviation events: suppressing the branching for a linear time, moving the Brownian particle to a linear distance, and emptying a ball with linear radius. The remaining three components, namely parts two, four and six are not large deviation events and do not give exponential costs.

\subsection{Proof of the upper bound}

Our approach is as follows. We first split the time interval $[0,t]$ into two pieces at the point  $\eta t$, $\eta \in [0,1]$, which is the instant when the population of skeleton particles exceed $\lfloor t^{d+\varepsilon} \rfloor$. We discretize the interval to pinpoint the optimal $\eta$. In order to obtain an upper bound, we ignore the possibility of trapping in the first piece, but take into account the cost of polynomial growth rather than the expected exponential growth of particles. In the second piece, given that we now have at least $\lfloor t^{d+\varepsilon} \rfloor +1$ particles in the system, we argue that at least one of these particles should have a trap-free ball of certain radius, say $\rho_t$, around it. For an upper bound, this gives us the combined cost of moving a Brownian particle to a distance $ct$, $c \geq 0$, during $[0,\eta t] $ and clearing the $\rho_t$-ball around it. To find the optimal $c$, we discretize the space. We note that this approach yields a double sum of terms, of which only the largest one contributes on an exponential scale.

Fix $\beta$, $d$, $f$. Let $0<\varepsilon<1$. Recall that $|Z(t)|$ is the number of particles in the system at time $t$, $|Z^*(t)|$ of which are skeleton particles. For $t>1$, define
\begin{equation} \eta_t=\sup\left\{\eta \in [0,1]:|Z^*(\eta t)|\leq \lfloor t^{d+\varepsilon} \rfloor \right\}.
\end{equation}
Before we proceed, we prove the following lemma:

\begin{lemma} \label{lemma8} For $n\in\left\{1,2,3,\ldots\right\}$ and $i\in\left\{0,1,2,\ldots,n-1\right\}$,
\begin{equation} P\left(\frac{i}{n}< \eta_t \mid \mathcal{E}^c\right)=\exp[-\beta \alpha \frac{i}{n}t+o(t)]. \nonumber
\end{equation}
\end{lemma}
\begin{proof} Let $n\in\left\{1,2,3,\ldots\right\}$ and $i\in\left\{0,1,2,\ldots,n-1\right\}$. By part 4 of Lemma~\ref{decomp}, we know that the process $|Z^*|$ conditioned on non-extinction of $|Z|$ is equal in law to a continuous time branching process with rate $\beta(1-f'(q))=\beta \alpha$ and offspring distribution $\bar{\lambda}$, where $\bar{\lambda}(0)=\bar{\lambda}(1)=0$. On the event $\mathcal{E}^c$, since the first particle is of infinite line of descent and $t>1$, in view of $\left\{\frac{i}{n}< \eta_t\right\}=\left\{|Z^*(\frac{i}{n}t)|\leq \lfloor t^{d+\varepsilon}\rfloor\right\}$, it follows that  $P(\frac{i}{n}< \eta_t \mid \mathcal{E}^c)\geq P(|Z^*(\frac{i}{n}t)|=1)=\exp(-\beta \alpha \frac{i}{n}t)$.

Now consider a strictly dyadic branching Brownian motion $\tilde{Z}$. We have $P(|\tilde{Z}(t)|=k)=e^{-\beta t}(1-e^{-\beta t})^{k-1}$ for every $k\in\left\{1,2,3,\ldots\right\}$ and for every $t\geq 0$, see \cite{KT1975}. It follows that $P(|\tilde{Z}(t)|> k)=(1-e^{-\beta t})^k$. Using the binomial expansion $(1-x)^n=1-nx+{n \choose 2}x^2+\ldots+(-1)^n x^n$, we obtain
\begin{align}
P(|& \tilde{Z}(t)|\leq \lfloor t^{d+\varepsilon}\rfloor) \nonumber \\
=&\ 1-(1-e^{-\beta t})^{\lfloor t^{d+\varepsilon}\rfloor} \nonumber \\
=&\ e^{-\beta t}\left(\lfloor t^{d+\varepsilon} \rfloor -{{\lfloor t^{d+\varepsilon}\rfloor} \choose 2}e^{-\beta t}+\ldots +(-1)^{\lfloor t^{d+\varepsilon}\rfloor-1}e^{-\beta t (\lfloor t^{d+\varepsilon}\rfloor-1)} \right) \nonumber \\
\leq &\ e^{-\beta t}\lfloor t^{d+\varepsilon} \rfloor=\exp[-\beta t+o(t)]. \nonumber
\end{align}
The result follows by comparison with the strictly dyadic case, since $\bar{\lambda}(0)=\bar{\lambda}(1)=0$.
\end{proof}

Henceforth, all probabilities written should be understood as conditioned on $\mathcal{E}^c$.
Now, for every $n\in\left\{1,2,3,\ldots\right\}$,
\begin{align} (&\mathbb{E}\times P)(T>t) \nonumber \\
=&\ \sum_{i=0}^{n-1} (\mathbb{E}\times P)\left(\left\{T>t\right\}\cap \left\{\frac{i}{n}\leq \eta_t <\frac{i+1}{n}\right\} \right)+(\mathbb{E}\times P)\left(\left\{T>t\right\}\cap \left\{\eta_t=1\right\} \right) \nonumber \\
\leq &\ \sum_{i=0}^{n-1} \exp\left[-\beta \alpha \frac{i}{n}t+o(t)\right](\mathbb{E}\times P^{(i,n)}_t)\left\{T>t\right\}+\exp[-\beta \alpha t+o(t)], \label{eq21}
\end{align}
where we use Lemma~\ref{lemma8}, and introduce the conditional probabilities
\begin{equation} P^{(i,n)}_t(\cdot)=P\left(\ \cdot \ \middle| \ \frac{i}{n}\leq \eta_t <\frac{i+1}{n}\right), \quad i=0,1,\ldots,n-1. \nonumber
\end{equation}
Recall that $\left\{\eta_t<(i+1)/n\right\}\subseteq\left\{|Z^*(t(i+1)/n)|>\lfloor t^{d+\varepsilon}\rfloor\right\}$. Let $A_t^{(i,n)}$, $i=0,1,\ldots,n-1$ be the event that among the $|Z^*(t(i+1)/n)|$ skeleton particles alive at time $t(i+1)/n$, there are $\leq \lfloor t^{d+\varepsilon} \rfloor$ particles such that the ball with radius
\begin{equation}
\rho_t^{(i,n)}:=(1-\varepsilon)\sqrt{2\beta m}\left(1-\frac{i+1}{n}\right)t \label{eqradius}
\end{equation}
around the particle is non-empty (i.e., contains a trap point).
Estimate
\begin{align} (\mathbb{E}\times P^{(i,n)}_t)(T>t)\leq (\mathbb{E}&\times P^{(i,n)}_t)(A_t^{(i,n)}) \nonumber \\
+&(\mathbb{E}\times P^{(i,n)}_t)(T>t\ |\ [A_t^{(i,n)}]^c). \label{eq22}
\end{align}
On the event $[A_t^{(i,n)}]^c$ there are $>\lfloor t^{d+\varepsilon} \rfloor$ balls containing a trap, and by Theorem~\ref{teo1} and ($\ref{eqradius}$), the sub-BBM emanating from the center of each ball exits this ball in the remaining time $(1-(i+1)/n)t$ with a probability tending to $1$ as $t\rightarrow \infty$. Abbreviate $m_t:=\lfloor t^{d+\varepsilon}\rfloor +1.$ With negligible error, we may assume that more than half of the $m_t$ sub-BBMs exit their respective balls that each contain a trap point, in the remaining time. More precisely, the probability that at least $\lfloor\frac{1}{2}m_t\rfloor$ of them stay inside their respective balls, is not more than\footnote{Here we used the trivial estimate
$$\binom{m}{\lfloor m/2\rfloor}\le 2^{m}.$$}
 $$2^{m_{t}}\cdot p_t^{\lfloor m_{t}/2\rfloor}=\exp(m_t \log 2 +\lfloor m_{t}/2\rfloor \log p_t)=\text{SES},$$ where $p_t$ is the probability that a single tree stays inside the ball, and SES stands for `superexponentially small in $t$.' (In the exponent, the second term dominates, because $p_t\to 0$.) Clearly, the probability that all trees avoid traps can be upper estimated by the probability that those trees that are exiting the balls avoid traps. Now, on the event that the number of exiting trees is more than $\lfloor m_{t}/2\rfloor$, we estimate as follows. 

Upon exiting, each one of the (at least $\lfloor m_{t}/2\rfloor$+1) sub-BBMs hits a trap inside this ball with a probability at least $C_1/[\rho_t^{(i,n)}]^{d-1}$, where we have assumed that the trap point is on the boundary of the $\rho_t^{(i,n)}$-ball for an upper bound on the trap-avoiding probability. Note that $C_1/[\rho_t^{(i,n)}]^{d-1}$ is just the ratio of the surface area of the ball that intersects the trap (given that the trap point is on the boundary) to the total surface area of the ball, and $C_1$ is a constant that depends on the dimension $d$ and the trap radius $a$. Hence, the second term on the right-hand side of (\ref{eq22}) is bounded above by
\begin{equation} \left[1-\frac{C_1}{[\rho_t^{(i,n)}]^{d-1}}\right]^{\lfloor m_t/2\rfloor+1}\leq \exp[-C_2 t^{1+\varepsilon}] =\text{SES} \nonumber
\end{equation}
uniformly in all parameters, where $C_2>0$ is another constant. Hence, the second term on the right-hand side of (\ref{eq22}) is superexponentially small.

Now consider the first term on the right-hand side of (\ref{eq22}). Since $g_d(rt,bt)=tg_d(r,b)$ for $r,b>0$, the cost of clearing a ball with radius linear in $t$ is exponentially small in $t$ as $t\rightarrow\infty$. Then, for large $t$, we have
\begin{align} (\mathbb{E}\times P^{(i,n)}_t)(A_t&^{(i,n)}\ |\ |Z^*(t(i+1)/n)|=\lfloor t^{d+\varepsilon}\rfloor+j) \nonumber \\ &\geq (\mathbb{E}\times P^{(i,n)}_t)(A_t^{(i,n)}\ |\ |Z^*(t(i+1)/n)|=\lfloor t^{d+\varepsilon}\rfloor+(j+1)) \nonumber
\end{align}
for every $j\in \left\{1,2,\ldots\right\}$. It follows that
\begin{align} (\mathbb{E}\times P^{(i,n)}_t)(A_t&^{(i,n)})\leq (\mathbb{E}\times P^{(i,n)}_t)(A_t^{(i,n)}\ |\ |Z^*(t(i+1)/n)|=\lfloor t^{d+\varepsilon}\rfloor+1). \label{eq23}
\end{align}
We continue to estimate from above as follows. Each skeleton particle alive at time $t(i+1)/n$ is at a random point, whose spatial distribution is identical to that of $W(t(i+1)/n)$, where $W$ denotes the standard Brownian motion. Let $x_0$ be the center of the empty ball at time $t(i+1)/n$ that is closest to the origin. Let $0\leq c<\infty$, $\delta>0$, and define $B_j$, $j=1,\ldots,\lfloor t^{d+\varepsilon}\rfloor+1$ to be the event that $j$th skeleton particle at time $t(i+1)/n$ is at a distance $ct\leq r<(c+\delta)t$ from the origin and that the $\rho_t^{(i,n)}$-ball around it is empty. Then,
\begin{align} (\mathbb{E}\times &P^{(i,n)}_t)(A_t^{(i,n)}\cap \left\{ct\leq |x_0|<(c+\delta)t\right\}\ |\ |Z^*(t(i+1)/n)|=\lfloor t^{d+\varepsilon}\rfloor+1) \nonumber \\
\leq & (\mathbb{E}\times P^{(i,n)}_t)(B_1\cup \ldots \cup B_{\lfloor t^{d+\varepsilon}\rfloor+1}) \nonumber \\
\leq & (\lfloor t^{d+\varepsilon}\rfloor+1)(\mathbb{E}\times P^{(i,n)}_t)(B_1) \nonumber \\
=& (\lfloor t^{d+\varepsilon}\rfloor+1)\exp\left[-\frac{c^2}{2(i+1)/n}t+o(t)\right] \nonumber \\
&\ \times \exp\left[-l g_d\left((1-\varepsilon)\sqrt{2\beta m}(1-\frac{i+1}{n}),c\right)t+O(\delta)t+o(t)\right], \label{eq24}
\end{align}
where we have used the independence of BBM and trap mechanisms, and Lemma~\ref{intensity} in passing to the last equality. Indeed, on the event $A_t^{(i,n)}$ conditioned on $|Z^*(t(i+1)/n)|=\lfloor t^{d+\varepsilon}\rfloor+1$, there is at least 1 skeleton particle with an empty ball around it. From (\ref{eq22}),(\ref{eq23}) and (\ref{eq24}), we obtain
\begin{align} (\mathbb{E}\times &P^{(i,n)}_t)(T>t) \nonumber \\
\leq & \sum_{j=0}^{n-1} (\mathbb{E}\times P^{(i,n)}_t)\left(\left\{\frac{j}{n}\sqrt{2\beta}t\leq |x_0|<\frac{j+1}{n}\sqrt{2\beta}t\right\}\cap A_t^{(i,n)}\right) \nonumber \\
&\ \ +(\mathbb{E}\times P^{(i,n)}_t)\left(\left\{|x_0|\geq \sqrt{2\beta}t\right\}\cap A_t^{(i,n)}\right)+\text{SES} \nonumber \\
\leq & \sum_{j=0}^{n-1} (\lfloor t^{d+\varepsilon}\rfloor+1) \exp\left[-\frac{\beta j^2/n^2}{(i+1)/n}t+o(t)\right] \nonumber \\
&\ \ \times \exp\left[-l g_d\left((1-\varepsilon)\sqrt{2\beta m}(1-\frac{i+1}{n}),\frac{j}{n}\sqrt{2\beta}\right)t+O(1/n)t+o(t)\right] \nonumber \\
&\ \ +\exp[-\beta t+o(t)]+\text{SES}. \label{eq25}
\end{align}
Here, the SES comes from the second term on the right-hand side of (\ref{eq22}). Also, in passing to the last inequality, we have used that the probability of the event $\left\{|x_0|\geq \sqrt{2\beta}t\right\}$ is bounded above by the probability that a single Brownian particle is at a distance $\geq \sqrt{2\beta}t$ at time $t$, which is $\exp[-\beta t+o(t)]$.

Substituting (\ref{eq25}) into (\ref{eq21}), and optimizing over $i,j\in\left\{0,1,\ldots,n-1\right\}$ gives
\begin{align} &\underset{t\rightarrow \infty}{\limsup}\frac{1}{t}\log(\mathbb{E}\times P)(T>t) \nonumber \\
\leq & -\underset{i,j\in\left\{0,1,\ldots,n-1\right\}}{\min}\left\{\frac{\beta \alpha i}{n}+\frac{\beta j^2/n^2}{(i+1)/n}+ l g_d\left((1-\varepsilon)\sqrt{2\beta m}(1-\frac{i+1}{n}),\frac{j}{n}\sqrt{2\beta}\right)\right\}. \nonumber
\end{align}
Now let $\eta=i/n$, $c=j/n\sqrt{2\beta}$, let $n\rightarrow \infty$, use the continuity of the functional form from which the minimum is taken, and finally let $\varepsilon \rightarrow 0$ to obtain the desired upper bound,
\begin{equation} \underset{t\rightarrow\infty}{\limsup}\frac{1}{t}\log(\mathbb{E}\times P)(T>t)\leq -I(l,f,\beta,d).  \nonumber
\end{equation}
Note that $c>\sqrt{2\beta}$ cannot minimize
\begin{equation} \left\{\beta\alpha\eta+\frac{c^2}{2\eta}+l g_d(\sqrt{2\beta m}(1-\eta),c)\right\} \label{eq26}
\end{equation}
since this makes (\ref{eq26}) greater than $\beta$, which can be improved by setting $c=0$ and $\eta=1$.
\qed

\section{Proof of Theorem~\ref{teo3}}

Let
\begin{equation} G_d(\eta,c)=\beta\alpha\eta+\frac{c^2}{2\eta}+l g_d(\sqrt{2\beta m}(1-\eta),c) \label{eq27}
\end{equation}
so that
\begin{equation} I(l,f,\beta,d)=\underset{\eta \in [0,1],c \in [0,\sqrt{2\beta}]}{\text{min}} G_d(\eta,c). \label{eq28}
\end{equation}
To see the existence of the minimizers $\eta^*$, $c^*$ of (\ref{eq28}), note that $G_d$ is lower semicontinuous on the compact subset $[0,1]\times[0,\sqrt{2\beta}]$ of $\mathbb{R}^2$ since
\begin{equation} \underset{(\eta,c)\rightarrow(0,0)}{\lim}\beta\alpha\eta+l g_d\left(\sqrt{2\beta m}(1-\eta),c\right)=l s_d \sqrt{2\beta m} \nonumber
\end{equation}
and $c^2/(2\eta)\geq 0$.

We refer the reader to \cite{E2003} for the proof of their uniqueness when $l\neq l_{cr}$.

Consider $d=1$. Since $g_1(r,b)=2r$, the minimum over $c$ in (\ref{eq28}) is taken at $c^*=0$, so that (\ref{eq28}) reduces to
\begin{align} I(l,f,\beta,d)=&\underset{\eta \in [0,1]}{\text{min}}\left\{\beta\alpha\eta+2l\sqrt{2\beta m}(1-\eta)\right\} \nonumber \\
=& \underset{\eta \in [0,1]}{\text{min}}\left\{\eta(\beta\alpha-2l\sqrt{2\beta m})+2l\sqrt{2\beta m}\right\}. \nonumber
\end{align}
This proves (\ref{eq51}) and (\ref{eq52}), and identifies $l_{cr}$ as the solution to $\beta\alpha=2l\sqrt{2\beta m}$.

Now consider $d\geq 2$. We have
\begin{align} G_d(\eta,c)-G_d(0,0)&=\beta\alpha\eta+\frac{c^2}{2\eta} \nonumber \\
+&l[g_d(\sqrt{2\beta m}(1-\eta),c)-g_d(\sqrt{2\beta m},0)], \label{eq29}
\end{align}
where the last term on the right-hand side is less than or equal to zero, with equality if and only if $(\eta,c)=(0,0)$ (this follows from the definiton of $g_d$). Suppose that $(\eta,c)=(0,0)$ is a minimizer when $l=l_0$. Then, by uniqueness of minimizers, the right-hand side of (\ref{eq29}) is strictly positive for all $(\eta,c)\neq (0,0)$ when $l=l_0$. As the last term in the right-hand side of (\ref{eq29}) is strictly negative for all $(\eta,c)\neq (0,0)$, it follows that for all $l<l_0$, the right-hand side of (\ref{eq29}) is zero when $(\eta,c)=(0,0)$ and strictly positive otherwise. Therefore, there must exist $l_{cr}\in (0,\infty)$ such that
\begin{enumerate}
	\item $(\eta,c)=(0,0)$ is the minimizer when $l<l_{cr}$;
	\item $(\eta,c)=(0,0)$ is not the minimizer when $l>l_{cr}$.
\end{enumerate}
Henceforth, we will take this as the definition of $l_{cr}$ for $d\geq 2$. This proves the first parts of (\ref{eq53}) and (\ref{eq54}). Note that $l_{cr}\in(0,\infty)$, because the last term on the right-hand side of (\ref{eq29}) decreases without bound as $l\rightarrow \infty$ and tends to zero as $l\rightarrow 0$ for fixed $(\eta,c)\neq (0,0)$.

Now let $d\geq 2$ and $l>l_{cr}$. We know that $(\eta,c)=(0,0)$ is not a minimizer. The combination $\eta^*=0$, $c^*>0$ is not possible because $G(0,c)=\infty$ for all $c>0$. The combination $\eta^*>0$, $c^*=0$ is ruled out as well because $G(\eta,0)$ takes its minimum either at $\eta=0$ or $\eta=1$, so $\eta^*>0$ would imply that $\eta^*=1$. However, $\eta^*=1$ can be excluded via \cite[Lemma 4.1]{E2003}. Also, $c^*=\sqrt{2\beta}$ is not possible since this yields $G_d> \beta$, whereas $(\eta,c)=(1,0)$ is not a minimizer yet yields $G_d\leq\beta$ since $\alpha\in(0,1]$. Hence, we conclude that the minimizers for $d\geq 2$ and $l>l_{cr}$ appear in the interior of $[0,1]\times[0,\sqrt{2\beta}]$. This proves the second parts of (\ref{eq53}) and (\ref{eq54}).

In the rest of the proof, we find $l_{cr}$ for $d\geq 2$.
For $R\geq0$, let
\begin{equation} g_d(R)=\int_{B(0,R)}\frac{\text{d}x}{|x+e|^{d-1}}. \label{fd}
\end{equation}
Then, we may write (\ref{eq27}) as
\begin{equation} G_d(\eta,c)=\beta \alpha \eta+\frac{c^2}{2\eta}+l c g_d\left(\frac{\sqrt{2\beta m}(1-\eta)}{c}\right). \label{gd}
\end{equation}
The stationary points are the solutions of the equations $\frac{\partial G_d}{\partial \eta}=0$ and $\frac{\partial G_d}{\partial c}=0$, which yield
\begin{align} 0=& \beta\alpha-\beta m v^2-l\sqrt{2 \beta m} g_d'(u) \nonumber \\
              0=& \sqrt{2 \beta m} v+l[g_d(u)-u g_d'(u)] \label{eq30}
\end{align}
upon setting
\begin{equation} u=\frac{\sqrt{2\beta m}(1-\eta)}{c},\quad v=\frac{c}{\sqrt{2\beta m}\eta}. \label{uv}
\end{equation}
Eliminating $g_d'$ in (\ref{eq30}) gives
\begin{equation} \frac{l}{\sqrt{2\beta m}}g_d(u)=-v+\frac{1}{2m}u(\alpha-m v^2). \label{eq31}
\end{equation}
It follows by (\ref{gd}), (\ref{uv}) and (\ref{eq31}) that for $d\geq 2$ and and $l>l_{cr}$,
\begin{equation} I(l,f,\beta,d)=G_d(u^*,v^*)=\beta(\alpha-m v^{*2}). \label{eq32}
\end{equation}
Since the minimizer $(u^*,v^*)$ must satisfy (\ref{eq31}) (note that there may be more than one pair $(u,v)$ that satisfies (\ref{eq31})), we conclude by (\ref{eq32}) that $v^*$ is the maximal value of $v>0$ on the curve in the $(u,v)$-plane given by (\ref{eq31}).

Recall the following lemma from \cite{E2003}, where a proof is given as well.
\begin{lemma} \label{last}
$R\mapsto g_d'(R)$ is strictly increasing on $(0,1)$, infinite at $1$, and strictly decreasing on $(1,\infty)$.
\end{lemma}

Using (\ref{lcr}), we may write (\ref{eq31}) as
\begin{equation}   \frac{g_d(u)l}{2s_d m l_{cr}^*}=-v+\frac{1}{2m}u(\alpha-m v^2). \label{eq33}
\end{equation}
Since $g_d(u)\sim s_d u$ as $u\rightarrow \infty$ by (\ref{fd}), we obtain from Lemma~\ref{last} that the left-hand side of (\ref{eq33}) is strictly convex on $(0,1)$, has infinite slope at $1$, is strictly concave on $(1,\infty)$, and has limiting derivative $l/(2 m l_{cr}^*)$ (see Figure 1).

\vspace{1cm}

\setlength{\unitlength}{0.4cm}

\begin{picture}(22,18)(0,-8)

  \put(0,0){\line(18,0){18}}
  \put(0,0){\line(0,12){12}}
  \put(0,0){\line(0,-3){3}}
  \put(-1,-.3){$0$}
  \put(19,-.3){$u$}
  {\thicklines
   \qbezier(0,-2)(6,2)(12,6)
   \qbezier(0,0)(4.5,0.2)(5,3)
   \qbezier(5,3)(5.1,5.5)(9,10)
  }
  \qbezier[30](0,3)(3,3)(5,3)
  \qbezier[30](5,0)(5,1.5)(5,3)
  \qbezier[90](0,0)(5,6)(10,12)
  \put(5,3){\circle*{.25}}
  \put(5,-1){$1$}
  \put(-1.5,-2.2){$-v$}
  \put(9.5,10){(1)}
  \put(12.5,6){(2)}

  \put(0,-5)
  {\small
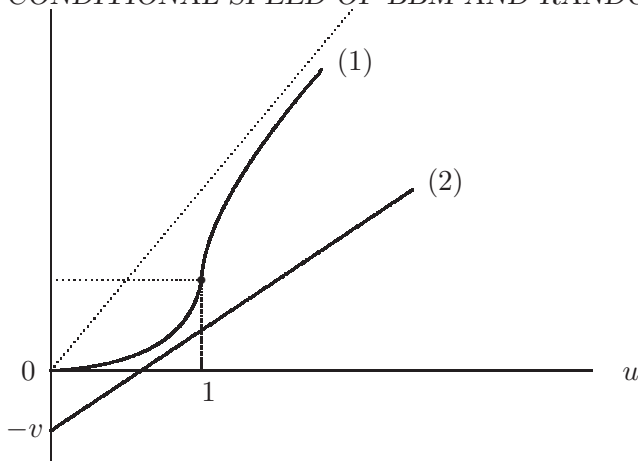
   Figure\ 1.~ Qualitative plot of: (1) $u\mapsto \frac{g_d(u)l}{2s_d m l_{cr}^*}$, (2) $u\mapsto -v+\frac{1}{2m}u(\alpha-m v^2)$.
  \normalsize
  }
  \put(0,-6)
  {\small
   The dotted line is $u\mapsto \frac{1}{2 m} \frac{l}{l^*_{cr}}u$.
  \normalsize
  }

\end{picture}

Firstly, see that $l_{cr}\leq \alpha l_{cr}^*$. Indeed, assume the contrary. Then, there exists $\bar{l}\in(\alpha l_{cr}^*,l_{cr})$, which implies that $(\eta,c)=(0,0)$ is the unique minimizer for $G_d$ when $l=\bar{l}$. However, $G_d(0,0)=\beta \bar{l}/l_{cr}^*>\beta\alpha=G_d(1,0)$, which contradicts $(\eta,c)=(0,0)$ being the minimizer. This implies that $l_{cr}\leq \alpha l_{cr}^*$. We claim that $l_{cr}$ can be identified by the following formula
\begin{equation} l_0:=\sup\left\{l>0:\frac{l}{l_{cr^*}}\frac{1}{2s_d m}g_d(u)>-\sqrt{\frac{1}{m}\left(\alpha-\frac{l}{l_{cr^*}}\right)}+\frac{1}{2m}\frac{l}{l_{cr^*}}u \quad \forall u\in(0,\infty)\right\}. \label{eq34}
\end{equation}

In order to prove our claim, i.e., $l_0=l_{cr}$, we need to show that $(\eta,c)=(0,0)$ is the minimizer for $G_d$ when $l<l_0$, and $(\eta,c)=(0,0)$ is not the minimizer when $l>l_0$. Equivalently, as $G_d(0,0)=\beta l/l_{cr}^*$, we need to show via (\ref{eq32}) that for $l>l_0$ there exists a $v>\sqrt{1/m\left(\alpha-l/l_{cr^*}\right)}$ such that $v$ satisfies (\ref{eq33}) for some $u>0$, and that for $l<l_0$ there is no such $v$. As we know that $l_{cr}>0$ and $(\eta,c)=(0,0)$ is the unique minimizer when $l<l_{cr}$, there must exist an $l_1>0$ such that the inequality in (\ref{eq34}) holds for all $u\in(0,\infty)$ when $l=l_1$, hence $0<l_1<l_0$. Now take any $0<l_2<l_1$, and see that the inequality in (\ref{eq34}) holds for all $u\in(0,\infty)$ also for $l=l_2$. This implies that if $l<l_0$, then the curve on the left-hand side of (\ref{eq33}) and the line on the right-hand side do not touch when $v=\sqrt{1/m\left(\alpha-l/l_{cr^*}\right)}$. As both terms on the right-hand side of (\ref{eq33}) decrease as $v$ increases and the left-hand side of (\ref{eq33}) does not depend on $v$, the curve on the left-hand side of (\ref{eq33}) and the line on the right-hand side do not touch for $v>\sqrt{1/m\left(\alpha-l/l_{cr^*}\right)}$ as well when $l<l_0$.

Now consider $l>l_0$. By definition of $l_0$, there exists $u\in (0,\infty)$, say $\bar{u}$, such that
\begin{equation} \frac{l_0}{l_{cr^*}}\frac{1}{2s_d m}g_d(u)=-\sqrt{\frac{1}{m}\left(\alpha-\frac{l_0}{l_{cr^*}}\right)}+\frac{1}{2m}\frac{l_0}{l_{cr^*}}u. \label{eq35}
\end{equation}
Note that since $g_d(u)<s_d u$ for all $u\in (0,\infty)$, (\ref{eq35}) implies that we have the strict inequality $l_0<\alpha l_{cr}^*$. Now take any $l_1\in(l_0,\alpha l_{cr}^*)$, and observe that at $u=\bar{u}$, the right-hand side of (\ref{eq35}) becomes greater than the left-hand side if we replace $l_0$ by $l_1$. Also, since $g_d(u)\sim s_d u$, the left-hand side of (\ref{eq35}) must be greater than the right-hand side for large $u$ if we replace $l_0$ by $l_1$. As the functions on both sides of (\ref{eq35}) are continuous with respect to $u$, it follows by the intermediate value theorem that there exists $u\in (0,\infty)$, say $\tilde{u}$, such that (\ref{eq35}) holds, with $l_0$ replaced by $l_1$. This in turn implies that $G_d(u,v)=G_d(\tilde{u},\sqrt{1/m(\alpha-l_1/l_{cr^*})})=\beta l_1/l_{cr}^*$, but since the minimizer is unique, this shows that $(\eta,c)=(0,0)$ is not the minimizer for $l_0<l_1<l_{cr}^*$. As we have shown that $(\eta,c)=(0,0)$ is not the minimizer when $l\in (l_0,l_{cr}^*)$, we conclude by definition of $l_{cr}$ that $(\eta,c)=(0,0)$ is not the minimizer for $G_d$ for all $l>l_0$. Hence, $l_0=l_{cr}$.

Now put
\begin{equation}  \hat{g}_d(u)=s_d u-g_d(u) \quad \quad M_d=\frac{1}{2s_d m}\underset{u\in(0,\infty)}{\max}\hat{g}_d. \nonumber
\end{equation}
Then, (\ref{eq34}) reads
\begin{equation} l_{cr}=\sup\left\{l>0:\sqrt{\frac{1}{m}\left(\alpha-\frac{l}{l_{cr}^*}\right)}>\frac{l}{l_{cr}^*}M_d\right\}.
\end{equation}
This implies that
\begin{equation} \sqrt{\frac{1}{m}\left(\alpha-\frac{l_{cr}}{l_{cr}^*}\right)}=\frac{l_{cr}}{l_{cr}^*}M_d, \nonumber
\end{equation}
which completes the proof of (\ref{eq55}) and (\ref{eq56}).

We finally show that $u^*\in(0,1)$ when $l>l_{cr}$. Note that for a fixed $v$, if the line on the right-hand side of (\ref{eq33}) cuts the curve on the left-hand side at more than one point, then by (\ref{eq32}), this $v$ cannot be $v^*$ since the minimizer is unique. Hence, we are looking for a $v$ value such that (\ref{eq33}) is satisfied for exactly one $u$ value, and this pair is $(u^*,v^*)$. This implies that when $v=v^*$, the line on the right-hand side of (\ref{eq33}) is tangent to the curve on the left-hand side at $u=u^*$. Since $v^*>\sqrt{1/m(\alpha-l/l_{cr^*})}$, the slope of the line on the right-hand side of (\ref{eq33}) is $<1/(2m)(l/l_{cr}^*)$. Then, since the left-hand side of (\ref{eq33}) has infinite slope at $u=1$, is concave on $(1,\infty)$, and decreases asymptotically to $1/(2m)(l/l_{cr}^*)$, it follows that $u^*\in(0,1)$. By (\ref{uv}), we conclude that $c^*>\sqrt{2\beta m}(1-\eta^*)$. This completes the proof of Theorem~\ref{teo3}.

\section*{Acknowledgement}
We would like to thank the anonymous reviewers for their valuable comments that helped to improve the manuscript.

\bibliographystyle{plain}

\end{document}